\documentclass{amsart}
\usepackage[utf8]{inputenc}

\usepackage{amssymb,mathrsfs,enumitem}
\usepackage{bbm,mathtools}
\usepackage{cleveref}
\setenumerate[1]{label={(\roman*)}} % Global setting
\allowdisplaybreaks

\providecommand{\bysame}{\leavevmode\hbox to3em{\hrulefill}\thinspace}
\providecommand{\MR}{\relax\ifhmode\unskip\space\fi MR }
\providecommand{\href}[2]{#2}

\newcommand{\R}{\ensuremath{\mathbb{R}}}
\newcommand{\Exp}{\ensuremath{\mathbb{E}}}
\newcommand{\F}{\ensuremath{\mathcal{F}}}
\newcommand{\Sc}{\ensuremath{\mathcal{S}}}

\newcommand{\Lone}{\ensuremath{\mathcal{L}^1}}

\newcommand{\E}{\ensuremath{\mathcal{E}}}

\newcommand{\D}{\ensuremath{\mathscr{D}}}

\newcommand{\inpr}[3][]{\left\langle#2 \,,\, #3\right\rangle_{#1}}

\theoremstyle{plain}
\newtheorem{thm}{Theorem}[section]

\newtheorem{lem}[thm]{Lemma}
\newtheorem{propn}[thm]{Proposition}
\theoremstyle{definition}
\newtheorem{rem}[thm]{Remark}
\newtheorem{defn}[thm]{Definition}
\newtheorem{eg}[thm]{Example}

\numberwithin{equation}{section}
\begin{document}

\title[Gaussian flows]{Characterizing Gaussian flows arising from It\={o}'s 
stochastic differential equations}
\author{Suprio Bhar}
\date{}
\address{Suprio Bhar, Indian Statistical Institute Bangalore Centre.}
\email{suprio@isibang.ac.in}
\begin{abstract}
We introduce and characterize a class of flows, which turn out to be Gaussian. 
This characterization allows us to show, using the Monotonicity inequality, 
that the transpose of the flow, for an extended class of initial conditions, is 
the unique solution of the SPDE introduced in Rajeev and Thangavelu (2008).
\end{abstract}

\maketitle
\section{Introduction}
It\={o}'s stochastic differential equations provide a concrete model for 
stochastic flows, on which topic there is a considerable 
literature (see \cite{MR866340, MR608026, MR774584, MR805125, MR517235, 
MR776981, MR876080, 
MR776984, MR1472487, MR540035, MR622556, MR685758, MR736036} and the 
references therein). In this paper, we study three 
interrelated properties of Gaussian flows arising as solutions of stochastic 
differential equations, viz.
\begin{equation}\label{sde}
dX_t = \sigma(X_t)\, dB_t + b(X_t)\, dt,\, t \geq 0.
\end{equation}
Let $L, A, L^\ast, A^\ast,X_t, Y_t$ be as in \cite{MR2373102}, with $r$ - the 
dimension of the Brownian motion there - equal to $d$. Let $\Sc(\R^d)$ 
be the space of real valued rapidly decreasing smooth functions on $\R^d$. Let 
$\Sc'(\R^d), \E'(\R^d)$ denote the space of tempered distributions on $\R^d$ 
and the space of compactly supported distributions on $\R^d$, respectively. 
Note 
that, for $\psi \in \E'(\R^d)$, $\{Y_t(\psi)\}$ satisfies 
\cite[equation (3.7)]{MR2373102}: a.s.
\[Y_t(\psi) = \psi + \int_0^t A^{\ast}(Y_s(\psi))\,.dB_s + 
\int_0^t L^{\ast}(Y_s(\psi))\,ds, \, \forall t \geq 0,\]
in some Hermite Sobolev space $\Sc_p(\R^d)$.\\
Firstly, we show that the Gaussian flows for
\begin{enumerate}
\item $\sigma$ is a real 
square matrix of order 
$d$,
\item $b(x):= \alpha + Cx,\, 
\forall x \in \R^d$ where $\alpha=(\alpha_1,\cdots,\alpha_d) \in \R^d$ and 
$C=(c_{ij})$ is a real 
square matrix of order $d$,
\end{enumerate}
correspond to those 
flows that depend 
`deterministically' on the initial condition (see 
Definition \ref{deterministic-dependence}). In particular, $\{X^0_t + 
e^{tC}x\}$ solves 
\eqref{sde}, where $\{X^x_t\}$ denotes the solution of \eqref{sde} with initial 
condition $X_0 = x$ (see Lemma \ref{soln-form}).\\
Secondly, as a consequence the map $x \mapsto \psi(X^x_t),\, 
\psi \in \Sc(\R^d)$ is in $\Sc(\R^d)$. The general question as to when this map 
(which is $C^\infty$ when $\sigma, b$ are smooth) is in $\Sc(\R^d)$ is to our 
knowledge, open. This property allows us to extend the map $Y_t = X^\ast_t$ 
defined on $\E'(\R^d)$ (see \cite[equation (3.3)]{MR2373102}), from 
$\E'(\R^d)\cap\Lone(\R^d) \to \E'(\R^d)$ to $\Lone(\R^d)\to \Sc'(\R^d)$ and 
that the processes $\{Y_t(\psi)\}, \psi 
\in 
\Lone(\R^d)$ satisfy the stochastic partial differential equation 
\cite[(3.7)]{MR2373102}. Taking expectation on both sides of this equation, we 
obtain the existence of solution of the Cauchy problem for $L^\ast$, viz.
\begin{equation}\label{Lstar-ivp}
\frac{d\psi(t)}{dt} = L^{\ast}\psi(t); \quad
\psi(0)=\psi
\end{equation}
Thirdly, we note that these flows are unique solutions of \cite[equation 
(3.7)]{MR2373102}. The uniqueness of solutions of these equations remains 
unresolved to date, to our knowledge. As observed 
in \cite{MR2373102}, the uniqueness follows from the Monotonicity inequality. 
This was proved for the pair of operators $(A^\ast, L^\ast)$ corresponding to 
Gaussian flows in our class, in \cite[Theorem 4.2 and Theorem 4.6]{dohs}. The 
same techniques also prove the uniqueness of the Cauchy problem for $L^\ast$.\\
In Section 2 of this paper we introduce the 
notion of flows depending deterministically on the initial condition (see 
Definition \ref{deterministic-dependence}) and prove characterization results 
(Theorem \ref{nasc-detr-init}, Proposition \ref{nasc-linear-init}). The 
results say 
that these type of 
diffusions correspond to the case when $\sigma$ is a constant and $b$ is in an 
affine form.\\
In Section 3, using the composition $x \mapsto \phi(X^x_t), \phi \in \Sc(\R^d)$ 
we define $\{Y_t(\psi)\}, \psi \in \Lone(\R^d)$ and using Monotonicity 
inequality (see Theorem \ref{AL-star-bnds}) show that it is the unique solution 
of a 
stochastic partial differential equation (see Theorem \ref{SDE-Y}). As an 
application 
of this result, we 
show $\psi(t) = \Exp\, Y_t(\psi)$ is the unique solution to 
\eqref{Lstar-ivp} with initial condition $\psi \in \Lone(\R^d)$ (see 
Theorem \ref{ivp-soln}).\\
In Proposition \ref{dist-Y}, we prove that
the tempered distribution $Y_t(\psi)$ is given by the 
integrable function $e^{-t\,tr(C)}\,\tau_{X(t,0)}\psi(e^{-tC}\cdot)$ 
where $tr(C)$ is the trace of the matrix $C$. This 
representation is similar to that obtained in \cite[Lemma 3.6]{MR3063763} for 
the solution of stochastic differential equations given by certain non-linear 
operators.

\section{Characterizing diffusions with the general solution depending 
deterministically on the initial 
condition}
In this section, we introduce a class of diffusions and prove some 
characterization results related to the flow of the diffusions. In particular, 
these flows turn out to be Gaussian flows. Let $(\Omega, \F, (\F_t), P)$ be a 
filtered complete probability space 
satisfying the usual conditions. Let $\{B_t\}$ be the standard $d$ 
dimensional $(\F_t)$ 
Brownian motion. Now consider the diffusion: 
\begin{equation}\label{diffusion-arbit-init}
dX_t = \sigma(X_t)\,dB_t+b(X_t)\,dt,\, \forall t \geq 0,
\end{equation}
where the 
coefficients $\sigma=(\sigma_{ij}), b=(b_i), 1 \leq i,j \leq d$ are Lipschitz 
continuous. Note that $\sigma, b$ satisfy a linear growth condition, i.e. 
there exists a constant $K > 0$ such that
\[|\sigma(x)|+|b(x)| \leq K(1+|x|),\, \forall x \in \R^d,\]
where $|\cdot|$ denotes the Euclidean norm in the appropriate spaces. For any 
$x \in \R^d$, let $\{X_t^x\}$ 
denote the solution of 
\eqref{diffusion-arbit-init} 
with 
$X_0=x$. We make a definition of the class of diffusions which we are 
interested in.
\begin{defn}\label{deterministic-dependence}
We say the general solution of the diffusion \eqref{diffusion-arbit-init} 
depends 
deterministically on the initial condition, if there exists a function 
$f:[0,\infty)\times \R^d\to\R^d$ such that for any $x \in \R^d$, we have a.s.
\begin{equation}\label{ddi}
X_t^x(\omega) = f(t,x)+X^0_t(\omega),\, t 
\geq 0.
\end{equation}
\end{defn}
Due to the linear growth of the coefficients $\sigma, b$, the first and second 
moments of $X^x_t$ exist for all $x,t$. If \eqref{ddi} is satisfied, then for 
all $x,t$ we have 
\[f(t,x) = \Exp X^x_t - \Exp X^0_t = x + \int_0^t \left[\Exp b(X^x_s) - \Exp 
b(X^0_s)\right]\, ds.\]
This observation gives the next result. The 
component functions of $f$ are denoted by $f_1,\cdots,f_d$.
\begin{lem}
If the general solution of the diffusion \eqref{diffusion-arbit-init} 
depends 
deterministically on the initial condition, then for all $(t,x)\in 
[0,\infty)\times
\R^d$, the partial derivative 
$\frac{\partial f}{\partial t}(t,x)$ exists and for every fixed $x\in \R^d$, 
the 
map $t \mapsto \frac{\partial f}{\partial t}(t,x)$ is continuous, where 
$\frac{\partial f}{\partial t}(t,x) = (\frac{\partial f_1}{\partial t}(t,x), 
\cdots, \frac{\partial f_d}{\partial t}(t,x))$.
\end{lem}

In Theorem \ref{nasc-detr-init} and Proposition \ref{nasc-linear-init} we 
characterize diffusions 
depending deterministically on the initial condition. In the first result we 
obtain 
the characterization under a non-degeneracy condition on $\sigma$ and 
smoothness assumptions on certain derivatives of $b$ and $f$ and in the second, 
we 
consider the case when $f$ is in a product form.
\begin{thm}\label{nasc-detr-init}
Let $\sigma,b$ be Lipschitz continuous functions. Suppose 
the 
following happen:
\begin{enumerate}[label=(\roman*)]
\item there exists an
$x \in \R^d$ such that the determinant of $(\sigma_{ij}(x))$ is not zero,
\item $b_i \in C^2(\R^d,\R), i=1,\cdots,d$ where $b=(b_1,\cdots,b_d)$,
\item for every fixed $x\in 
\R^d$, 
the 
map $t \in [0,\infty) \mapsto \frac{\partial f}{\partial t}(t,x)$ is of 
bounded variation.
\end{enumerate}
Then the general solution of the diffusion \eqref{diffusion-arbit-init} depends 
deterministically on the initial condition through \eqref{ddi} 
if and only if $\sigma$ is 
a real non-singular matrix of order $d$ and $b$ is of the 
form $b(x) = 
\alpha + C x$ and $f(t,x) = e^{tC}x$ where 
$\alpha \in \R^d$ and $C$ is a real square matrix of order $d$.
\end{thm}
\begin{proof}
First we prove the necessity part. For any $x \in 
\R^d$, a.s. $t 
\geq 0$
\begin{align*}
f(t,x) &= X_t^x-X_t^0\\
&=f(0,x) +\int_0^t\Big[\sigma(X_s^x)-\sigma(X_s^0)\Big].\,dB_s
+ \int_0^t \Big[b(X_s^x)-b(X_s^0)\Big]\,ds.
\end{align*}
Note that necessarily we must have $f(0,x)=x$. Now rewriting above relation
\[\int_0^t\left[\sigma(X_s^x)-\sigma(X_s^0)\right].\,dB_s+\int_0^t 
\left[b(X_s^x)-b(X_s^0)-\frac{\partial f}{\partial t}(s,x)\right]\,ds = 0.\]
Both the terms are $0$ a.s. (see \cite[Chapter 
III, Theorem 12]{MR2020294}). By looking at the continuous integrand of 
the quadratic variation of the martingale, we have 
for all $x 
\in \R^d, i,j = 1,\cdots,d$ 
a.s. $\left[\sigma_{ij}(X_t^x)-\sigma_{ij}(X_t^0)\right]^2 = 0,\, \forall 
t\geq 0$. Putting $t=0$ we have $\sigma_{ij}(x)=\sigma_{ij}(0),\,x \in \R^d$ 
i.e. 
$\sigma$ is a 
constant $d\times d$ matrix and by our assumption, its determinant 
is non-zero.\\
On the other hand, for each $x\in \R^d$, a.s. $t \geq 0$,
\begin{equation}\label{condn-b-gen-f}
b(X_t^x)-b(X_t^0)-\frac{\partial f}{\partial t}(t,x) = 0.
\end{equation}
Evaluating at $t=0$ 
yields $b_i(x) = b_i(0)+\frac{\partial f_i}{\partial t}(0,x),\, i=1,\cdots,d$.
Let $ \{B^{(i)}_t\}$ denote the $i$th 
component 
of $\{B_t\}$. Since 
$b_i \in C^2(\R^d,\R)$, by It\={o} formula and equation 
\eqref{condn-b-gen-f}, we have
\begin{equation}\label{2nd-level-b-condn}
\begin{split}
\frac{\partial f_i}{\partial t}(t,x) - \frac{\partial f_i}{\partial t}(0,x) 
&=\sum_{j,k=1}^d 
\int_0^t  
\big(\partial_j b_i(X^x_s)-\partial_j b_i(X^0_s)\big)\,\sigma_{jk} 
dB^{(k)}_s\\
&+ \sum_{j=1}^d \int_0^t \left[\partial_jb_i(X^x_s)b_j(X^x_s) - 
\partial_jb_i(X^0_s)b_j(X^0_s) \right]\, ds\\
&+ \frac{1}{2} \sum_{j,k = 1}^d \int_0^t (\sigma 
\sigma^t)_{jk}\left[\partial_j\partial_kb_i(X^x_s) - 
\partial_j\partial_kb_i(X^0_s)\right]\, ds.
\end{split}
\end{equation}
Again, the martingale term must be zero and therefore so is its quadratic 
variation. Then for any 
$i,k=1,\cdots,d$ we have
a.s. $\sum_{j=1}^d\sigma_{jk}^2\big(\partial_j b_i(X^x_t)-\partial_j 
b_i(X^0_t)\big)^2 = 
0,\, t \geq 0$. Evaluating at $t=0$ and simplifying we have
$\sum_{j=1}^d\sigma_{jk}(\partial_j b_i(x)-\partial_j b_i(0)) = 0$. The last 
equation we can write as
\[\sigma^t \begin{pmatrix}
\partial_1 b_i(x)-\partial_1 b_i(0)\\
\cdots\\
\partial_d b_i(x)-\partial_d b_i(0)
\end{pmatrix}
=0.\]
Since $\sigma$ is non-singular, for each $i,j=1,\cdots,d$ the function 
$x \mapsto \partial_j b_i(x)$ is a constant function. Define 
$c_{ij}:=\partial_j b_i(0)$ and write $C=(c_{ij})$. Then
\begin{align*}
b_i(x)-b_i(0)
&=\int_0^{x_1} \partial_1 b_i(y,x_2,\cdots,x_d)\,dy+\int_0^{x_2} \partial_1 
b_i(0,y,x_3,\cdots,x_d)\,dy\\
&+\cdots+\int_0^{x_d} 
\partial_d b_i(0,\cdots,0,y)\,dy=\sum_{j=1}^d c_{ij}x_j
\end{align*}
Now for any fixed $x \in \R^d$, we have $f(0,x)=x$ and
\[f(t,x)-x=\int_0^t 
\left[b(X_s^x)-b(X_s^0)\right]\,ds=\int_0^t 
C\left[X_s^x-X_s^0\right]\,ds=\int_0^t 
Cf(s,x)\,ds.\]
Since for each $x$, the map $t \mapsto f(t,x)$ is continuous, we have $f(t,x) = 
e^{tC}x$. It is easy to check that for above $\sigma, b, f$ the finite 
variation term in \eqref{2nd-level-b-condn} vanishes.\\
The converse part can be verified through direct computation.
\end{proof}
In Definition \ref{deterministic-dependence}, if the function $f$ is in a 
product form, 
then a similar characterization can be obtained without additional smoothness 
assumptions on $b, f$.
\begin{propn}\label{nasc-linear-init}
Let $\sigma,b$ be Lipschitz continuous functions.
\begin{enumerate}[label=(\roman*)]
\item Suppose the general solution of the diffusion 
\eqref{diffusion-arbit-init} 
depends 
deterministically on the initial condition, where the function $f$ has the 
decomposition $f(t,x) = g(t)h(x)$ with $g\in 
C^1([0,\infty),\R), 
h:\R^d \to \R^d$. Then $f(t,x) = \tilde g(t)x$ for some $\tilde g \in \D$ 
where
\[\D := \{g \in C^1([0,\infty),\R): g(0)=1\}.\]
\item The solution to \eqref{diffusion-arbit-init} depends deterministically on 
the initial 
condition in the following form: there exists $g \in \D$ such that for each $x 
\in \R^d$, a.s. $t \geq 0$
\begin{equation}\label{linear-init}
X_t^x = g(t)x + X_t^0,
\end{equation}
if and only if $\sigma$ is 
a constant $d\times d$ matrix, $b(x) = 
\alpha + \beta x$ and $g(t)=e^{\beta t},\, t \geq 0$ where 
$\alpha \in \R^d,\beta \in \R$. In this 
case, the solution has the form
\[X_t^x = 
\begin{cases}
e^{\beta t}x+\sigma\int_0^t e^{\beta(t-s)}\, dB_s 
+ \frac{e^{\beta t}-1}{\beta}\alpha,\,\text{if}\; \beta \neq 0\\
x+t\alpha + \sigma B_t,\,\text{if}\; \beta = 0.
\end{cases}
\]
\end{enumerate}
\end{propn}
\begin{proof}
Since a.s. $X_0^x=f(0,x)+X^0_0$, we have
$x=g(0)h(x),\,\forall x 
\in \R^d$. So  $g(0)\neq 0$. Without loss of 
generality, we may 
assume $g(0)=1$. Then $h(x) 
= x,f(t,x)=g(t)x$. This proves part (i).\\
If \eqref{linear-init} holds for some $g \in \D$, then as in 
Theorem \ref{nasc-detr-init}, we can show $\sigma$ is a 
constant $d\times d$ matrix. Using \eqref{condn-b-gen-f}, we have for all $x 
\in \R^d$
\begin{equation}\label{condn2-b-f}
\text{a.s.}\quad b(X_t^x)-b(X_t^0)-g'(t)x = 0,\, \forall t\geq 0.
\end{equation}
Putting $t=0$ we have for all $x \in \R^d$, 
$b(x) = b(0)+g'(0)x$. So $b$ is determined by 
the 
values $b(0),g'(0)$. Using this affine form of $b$ in \eqref{condn2-b-f} we 
get the differential equation
\begin{equation*}
g'(0)g(t) = g'(t);\, t \geq 0; \quad
g(0)=1.
\end{equation*}
The solution is given by $g(t) = 
e^{g'(0)t}, t\geq 0$.\\
The converse part can be verified through direct computation.
\end{proof}

In dimension $d = 1$, for convex functions we can apply the following 
generalization of It\={o} formula.
\begin{thm}[{\cite[Chapter VI, (1.1) 
Theorem]{MR1725357}}]
If $\{X_t\}$ is a continuous real valued semimartingale and $f:\R\to\R$ is a 
convex function, then there exists a continuous increasing process $\{A^f_t\}$ 
such that a.s. $t \geq 0$
\[f(X_t) = f(X_0) + \int_0^t f'_{-}(X_s)\, dX_s + \frac{1}{2}A^f_t\]
where $f'_{-}$ is the left-hand derivative of $f$.
\end{thm}
Using the previous theorem, we get the following version of 
Theorem \ref{nasc-detr-init}.
\begin{propn}
Let $\sigma,b$ be Lipschitz continuous functions on $\R$. 
Suppose 
the 
following happen:
\begin{enumerate}[label=(\roman*)]
\item there exists an
$x \in \R$ such that $\sigma(x)$ is not zero,
\item $b$ is continuously differentiable and is a finite linear combination of 
convex functions,
\item for every fixed $x\in 
\R^d$, 
the 
map $t \in [0,\infty) \mapsto \frac{\partial f}{\partial t}(t,x)$ is of 
bounded variation.
\end{enumerate}
Then the general solution of the diffusion \eqref{diffusion-arbit-init} depends 
deterministically on the initial condition through \eqref{ddi} 
if and only if $\sigma$ is 
a non-zero constant function and $b$ is of the 
form $b(x) = 
\alpha + C x$ and $f(t,x) = e^{tC}x$ where 
$\alpha, C \in \R$.
\end{propn}

In the next proposition, we present an example where 
Definition \ref{deterministic-dependence} appears. Given a random field 
$(X^x_t, x \in 
\R^d, t \geq 0)$ in many situations 
it is reasonable to assume that the field can be decomposed as $X^x_t = Y^x_t + 
Z_t$, where $\{Y^x_t\}$ is a `local' component and $\{Z_t\}$ is a global 
component. We show that under certain conditions the `local' component has to 
be deterministic.
\begin{propn}\label{deterministic-local}
Suppose that $Z_t = X^0_t$ and that for all $x \in \R^d$, the field $Y^x_t = 
X^x_t - X^0_t$ is independent of $Z$. In addition assume that $\{X^x_t\}$ solves
\[dX_t = \sigma(X_t)\, dB_t + b(X_t)\, dt, t \geq 0;\, X_0 = x\]
and the 
sigma-fields generated by the processes $\{X^0_t\}$ and $\{B_t\}$ are the same. 
Then $\{Y^x_t\}$ 
is deterministic. \end{propn}
\begin{proof}
Under our hypothesis, $\{Y^x_t\}$ is both adapted to the said sigma-field and 
is 
independent of it. Hence $\{Y^x_t\}$ is deterministic.
\end{proof}
\begin{eg}
If $\sigma(x)$ is non-degenerate for all $x \in \R^d$, the 
sigma-fields generated by the processes $\{X^0_t\}$ and $\{B_t\}$ are the 
same. We present an example which shows that the local part may not be 
deterministic, if $\sigma(x)$ is degenerate for some $x$. Consider 
the stochastic 
differential equations in dimension one:
\[dX_t = x\, dB_t +(\alpha -X_t)\, dt;\, X_0 = \frac{x^2}{2},\]
where $\alpha$ is some fixed real number. The solution is given by
\[X^x_t = 
e^{-t}\frac{x^2}{2}+ x \int_0^t e^{-(t-s)}\, dB_s -\alpha(e^{-t} -1),\]
which is not of the form \eqref{ddi}, but the flow is Gaussian.
\end{eg}
In Proposition \ref{deterministic-local}, we can allow $\sigma, b$ to be 
random, but 
independent of $\{B_t\}$ and then the conclusion still holds, conditional on 
the $\sigma$-fields of $\sigma, b$. We take this to be in a product form in 
the next theorem.
\begin{thm}
Let $(\Omega',\F',P')$ be a complete probability space 
and $(\Omega'',\F'',(\F''_t),P'')$ a filtered complete probability space 
satisfying the usual conditions. Define $\Omega := \Omega' \times \Omega''$. 
Consider the filtered probability space $(\Omega, 
\F'\otimes\F'',(\F'\otimes\F''_t), P'\times P'')$. Let $\{B_t\}$ be an 
$(\F''_t)$ Brownian motion. Assume that $\F''_t = \sigma\{B_s : 0 \leq s \leq 
t\}$ and $\F'' = \sigma\{B_t : t \geq 0\}$. Let $b:\R^d\times\Omega \to 
\R^d$ be $\mathcal{B}(\R^d)\otimes\F'\otimes\F''_0/\mathcal{B}(\R^d)$ 
measurable 
and $\sigma:\R^d\times\Omega \to 
\R^{d\times d}$ be 
$\mathcal{B}(\R^d)\otimes\F'\otimes\F''_0/\mathcal{B}(\R^{d\times d})$ 
measurable, where $\mathcal{B}(\R^d)$ denotes the Borel sigma field on $\R^d$. 
Suppose that a unique strong solution to the following stochastic differential 
equation
\[dX_t = \sigma(X_t)\, dB_t + b(X_t)\, dt, t \geq 0;\, X_0 = x\]
exists for each $x \in \R^d$. Denote the solution by $\{X^x_t\}$. Suppose that
\begin{enumerate}
\item $\sigma\{B_t : t \geq 0\} = \sigma\{X_t^0 : t \geq 0\}$,
\item $\{X_t^x - X_t^0 : x \in \R^d, t \geq 0\}$ and $\{B_t : t \geq 0\}$ 
are independent.
\end{enumerate}
Then a.s. $\omega' \,(P')$, a.s. $\omega'' \,(P'')$ we have the process 
$\{X_t^x - X_t^0\}$ depends on $\omega'$ alone.
\end{thm}
\begin{proof}
By condition $(ii)$, a.s. $\omega' \,(P')$, $\{B_t: t \geq 0\}$ and 
$\{X^x_t(\omega',\cdot) - X^0_t(\omega',\cdot): t \geq 0, x \in \R^d\}$ are 
independent.\\
Since $\{X^x_t\}$ is the strong solution of a stochastic differential equation, 
there exists a $P'$-null set $\mathcal{N}'\subset \Omega'$ such that for every 
$\omega \in \Omega' \setminus \mathcal{N}'$, a.s. 
$\omega'' \,(P'')$,
\[X_t^x(\omega', \omega'') = x + \left(\int_0^t \sigma(X_s)\, 
dB_s\right)(\omega', \omega'') + \int_0^t b(X_s(\omega', \omega''), \omega', 
\omega'')\, ds,\, t \geq 0.\]
Hence a.s. $\omega' \,(P')$, the random variables $X^x_t(\omega',\cdot), t \geq 
0, x \in \R^d$ are measurable with respect to $\sigma\{B_t : t \geq 0\}$ and by 
condition $(i)$, so are $X^x_t(\omega',\cdot) - 
X^0_t(\omega',\cdot), t \geq 0, x \in \R^d$.\\
Hence a.s. $\omega' \,(P')$, $X^x_t(\omega',\omega'') - 
X^0_t(\omega',\omega''), t \geq 0, x \in \R^d$ is deterministic in $\omega''$, 
i.e. the random variables depend on $\omega'$ alone.
\end{proof}

\begin{rem}
\begin{enumerate}
\item One may formulate and prove similar results for the following type of 
condition
\[X_t^{s,x}(\omega) = f(t,s,x)+X^{s,0}_t(\omega),\, t 
\geq s; X^{s,x}_s = x\]
for $s \geq 0, x \in \R^d$.
\item If equation \eqref{ddi} holds, then $f(t,x) = \Exp [X_t^x - 
X^0_t]$. As such the conditions on $f$ (in Theorem 
\ref{nasc-detr-init}, Proposition \ref{nasc-linear-init}) can be stated in 
terms of $\Exp X^x_t, x \in \R^d$.
\item Diffusions satisfying \eqref{ddi} also satisfy the 
following condition: for any $x,y \in \R^d$, a.s. $t \geq 0$
\[X^x_t - X^y_t = f(t,x) - f(t,y).\]
In certain situations such differences were shown to be diffusions (see 
\cite[Proposition 2.2]{MR2531088}).
\item Semimartingales with independent increments have been considered in 
\cite[Chapter II]{MR1943877}. In particular, it was shown that any rcll process 
with independent increments must be a sum of a semimartingale with independent 
increments and a deterministic part (\cite[Chapter II, 5.1 
Theorem]{MR1943877}). This is similar to \eqref{ddi}, but we are interested in 
the dependence of a possible deterministic part of the flows (generated by 
stochastic differential equations) on the initial condition.
\end{enumerate}

\end{rem}

\section[Probabilistic representations of solutions of the Forward 
Equations]{Probabilistic representations of the solutions of the 
Forward equations}
Let $(\Omega,\F,(\F_t),P)$ be a filtered complete probability space satisfying 
the usual conditions and let $\{B_t\}$ denote the standard $d$ 
dimensional $(\F_t)$ Brownian 
motion. Let $\Sc'(\R^d)$ be the space of tempered distributions, which is the 
dual of the space $\Sc(\R^d)$ of real valued rapidly decreasing smooth 
functions on $\R^d$. Let $\Lone(\R^d)$ denote the space of integrable 
functions. Note that $\Lone(\R^d) \subset \Sc'(\R^d)$.\\
We obtain probabilistic 
representations of the solutions of \eqref{Lstar-ivp} where
\begin{enumerate}[label=(\roman*)]
\item $\sigma$ is a real 
square matrix of order 
$d$ and $b(x):= \alpha + Cx,\, 
\forall x \in \R^d$ where $\alpha=(\alpha_1,\cdots,\alpha_d) \in \R^d$ and 
$C=(c_{ij})$ is a real 
square matrix of order $d$,
\item $\psi \in \Lone(\R^d)$. \end{enumerate}

Since $\sigma, b$ are $C^\infty$ functions with bounded derivatives, there 
exists a diffeomorphic modification of 
the solution of
\begin{equation}\label{diffusion}
dX_t= \sigma(X_t).dB_t + b(X_t)dt; \quad
X_0=x,
\end{equation}
(see \cite{MR1472487}, \cite[Theorem 
2.1]{MR2373102}). We first observe that such a modification 
can be written in an explicit form, with a null set 
independent of deterministic initial conditions of \eqref{diffusion}.
\begin{lem}\label{soln-form}
Let $\sigma, b$ be as above. Let $\{X(t,0)\}$ denote the solution of 
\eqref{diffusion} with 
initial condition $X_0 = 0$. Then a.s. for all $t \geq 0, x \in \R^d$
\begin{equation}\label{affine-solution}
X(t,0) + e^{tC}x = x + \int_0^t \sigma\, dB_s + \alpha t + \int_0^t 
C(X(s,0)+e^{sC}x)\,ds
\end{equation}
so that the sum $\{X(t,0) + e^{tC}x\}$ solves the stochastic differential 
equation
\[dX_t= \sigma(X_t).dB_t + b(X_t)dt; \quad
X_0=x.\]
\end{lem}

\begin{proof}
The proof follows 
from the observation that $\sigma(X(t,0)) = \sigma(X(t,0) + 
e^{tC}x) = \sigma$ for any $t \geq 0, x \in \R^d$.
\end{proof}
For the case $\sigma = Id$ ($Id$ denotes the $d\times 
d$ identity 
matrix), $b(x)=-x$, we get the well-known Ornstein-Uhlenbeck diffusion.\\
Let $(\Sc_p(\R^d),\|\cdot\|_p), p \in \R$ be the Hermite Sobolev spaces (see 
\cite[Chapter 1.3]{MR771478}), which are real separable Hilbert spaces and are 
completions of $(\Sc(\R^d),\|\cdot\|_p)$.
Given $\psi \in \Sc(\R^d)$ (or $\Sc_p(\R^d)$) and $\phi \in \Sc'(\R^d)$
(or
$\Sc_{-p}(\R^d)$), the action of $\phi$ on $\psi$ will be denoted by
$\inpr{\phi}{\psi}$. For each $x$, let $\delta_x$ denote the Dirac 
distribution supported at the point 
$x$. The following result lists some properties of the Dirac 
distributions.
\begin{propn} \label{dirac-prpty}
\begin{enumerate}[label=(\roman*),ref=\ref{dirac-prpty}(\roman*)]
\item\label{dirac-limit} ({\cite[Theorem 
4.1]{MR2373102}}) Let $x \in \R^d$. Then $\delta_x \in 
\Sc_{-p}(\R^d)$ 
for any $p > \frac{d}{4}$. 
Further if $p > 
\frac{d}{4}$, then
$\lim_{|x|\to\infty}\|\delta_x\|_{-p} = 0$ and there exists a constant $R=R(p) > 
0$ 
such that 
$\|\delta_x\|_{-p} \leq R,\,\forall x \in \R^d$.
\item\label{delta-x-continuity} Let $p > \frac{d}{4}$. Then the map $x \in \R^d 
\mapsto \delta_x \in \Sc_{-p}(\R^d)$ is continuous.
\end{enumerate}
\end{propn}
\begin{proof}
Only part $(ii)$ requires a proof. In 
\cite[Proposition 
3.1]{MR2373102} it was shown that the map $x \in \R^d \mapsto \tau_x\phi \in 
\Sc_q(\R^d)$ is continuous where
\begin{enumerate}[label=(\alph*)]
\item $\phi \in \Sc_q(\R^d)$ for some real number $q$,
\item $\tau_x:\Sc'(\R^d)\to \Sc'(\R^d), x \in \R^d$ are translation operators, 
which on functions act as follows: $(\tau_xf)(y):=f(y-x),\, y \in \R^d$.
\end{enumerate}
But $\delta_0 \in 
\Sc_{-p}(\R^d)$ (by part $(i)$) and $\tau_x \delta_0 = 
\delta_x$. Hence the required continuity follows.
\end{proof}

\begin{lem}\label{Lone-and-Scp} Let $p > \frac{d}{4}$. Then
$\Lone(\R^d)\subset \Sc_{-p}(\R^d)$. 
\end{lem}
\begin{proof}
For $\psi \in 
\Lone(\R^d)$, $\int_{\R^d}|\psi(x)|.\|\delta_x\|_{-p} \,dx \leq R 
\int_{\R^d}|\psi(x)|\,dx < \infty$,
where $R= R(p)$ is the constant given by Proposition \ref{dirac-limit}. Hence 
$\int_{\R^d}\psi(x)\delta_x \,dx$ is a well-defined element of 
$\Sc_{-p}(\R^d)$. It can be checked that as a tempered distribution 
$\psi = \int_{\R^d}\psi(x)\delta_x \,dx$ 
and hence $\psi \in \Sc_{-p}(\R^d)$.
\end{proof}
In what follows, $\{X(t,x)\}$ and $\mathcal{N}$ will denote the solution and 
the null set mentioned in 
Lemma \ref{soln-form} respectively. As in \cite[equation (3.3)]{MR2373102}), 
for any 
$\psi \in \Lone(\R^d)$ we define
\begin{equation}\label{defn-Y}
Y_t(\omega)(\psi):= \int_{\R^d} \psi(x)\delta_{X(t,x,\omega)} \,dx,\, \omega 
\in 
\Omega\setminus\mathcal{N}
\end{equation}
and set $Y_t(\omega)(\psi) := 0$, if $\omega\in\mathcal{N}$.
\begin{propn}
Let $\psi, \{X(t,x)\}, \{Y_t(\psi\})$ be as above. Let $p > \frac{d}{4}$. Then 
$\{Y_t(\psi)\}$ is an $(\F_t)$ adapted $\Sc_{-p}(\R^d)$ valued continuous 
process. Furthermore, $Y_t(\psi)$ is norm-bounded, where the bound can be 
chosen to be independent of $t$.
\end{propn}

\begin{proof}
Since $\psi \in \Lone(\R^d)$ and $\delta_x, x \in \R^d$ are uniformly bounded 
in $\Sc_{-p}(\R^d)$ (Proposition \ref{dirac-limit}), $Y_t(\psi)$ is a 
well-defined element of $\Sc_{-p}(\R^d)$ for any $p > 
\frac{d}{4}$. Existence of a bound of $Y_t(\psi)$ independent of $t$ follows 
from the bound on $\delta_x, x \in \R^d$.\\
Since $\{X(t,x)\}$ is $(\F_t)$ adapted for each $x$, so is 
$\{Y_t(\psi)\}$. Since $\{X(t,x)\}$ has 
continuous paths for each $x \in \R^d$, by Proposition 
\ref{delta-x-continuity}, $\{\delta_{X(t,x)}\}$ also has continuous paths in 
$\Sc_{-p}(\R^d)$ for each $x \in \R^d$. Continuity of 
$\{Y_t(\psi)\}$ follows from the Dominated Convergence theorem.
\end{proof}

For 
any $t\geq 0, \omega\in \Omega\setminus\mathcal{N}$, $x \mapsto 
X(t,x,\omega)$ is affine and hence the map $x 
\mapsto 
\phi(X(t,x,\omega))$ is in $\Sc(\R^d)$ whenever $\phi \in \Sc(\R^d)$. Consider 
the linear map $X_t(\omega): \Sc(\R^d) \to \Sc(\R^d)$ defined by 
$(X_t(\omega)\phi)(x):=\phi(X(t,x,\omega)),\, x \in \R^d$. The following 
property of the map $X_t(\omega)$ will be used in our analysis.
\begin{lem}\label{composition-defined}
Fix any $t\geq 0, \omega\in \Omega\setminus\mathcal{N}$. The linear map 
$X_t(\omega): \Sc(\R^d) \to \Sc(\R^d)$ is continuous.
\end{lem}
\begin{proof}
Certain seminorms given by supremums on $\R^d$ generates the topology on 
$\Sc(\R^d)$ (see 
\cite[Chapter 1.3]{MR771478}). Since $x \mapsto 
X(t,x,\omega)$ is affine, we can establish necessary estimates in terms of the 
said seminorms.
\end{proof}
Let 
$X^{\ast}_t(\omega):\Sc'(\R^d)\to \Sc'(\R^d)$ denote the transpose of the map 
$X_t(\omega)$. Then for any $\theta \in \Sc'(\R^d)$,
\[\inpr{X^{\ast}_t(\theta)}{\phi} = \inpr{\theta}{X_t(\phi)},\, 
\forall \phi \in \Sc(\R^d).\]
Using \eqref{defn-Y}, for any $\phi \in \Sc(\R^d),\psi \in \Lone(\R^d)$ we have
\[\inpr{Y_t(\psi)}{\phi} = \int_{\R^d}\psi(x)\,\phi(X(t,x))\, 
dx=\int_{\R^d}\psi(x)\,(X_t(\phi))(x)\, dx=\inpr{\psi}{X_t(\phi)}.\]
This implies
\begin{equation}\label{YasX*}
Y_t(\psi) = X^{\ast}_t(\psi).
\end{equation}
Consider the derivative maps $\partial_i:\Sc(\R^d)\to
\Sc(\R^d)$ for $i=1,\cdots,d$ as in \cite{MR2373102}. On $\Sc(\R^d)$ consider 
the multiplication operators $M_i, 
i=1,\cdots,d$ 
defined by $(M_i\phi)(x):= x_i\phi(x),\, \phi \in \Sc(\R^d), x=(x_1,\cdots,x_d) 
\in 
\R^d$. By duality these operators can be extended to 
$M_i:\Sc'(\R^d)\to\Sc'(\R^d)$.\\
Let $\{e_1,\cdots,e_d\}$ be the standard basis vectors in 
$\R^d$. Let $n 
=
(n_1,\cdots,n_d) \in {\mathbb Z}_+^d := \{n= (n_1,\cdots, 
n_d) : n_i\,\text{are 
non-negative integers}\}$. Let $h_n, n \in \mathbb{Z}^d_+$ be 
the Hermite functions (see 
\cite[Chapter 1.3]{MR771478}). Note that $\{(2|n|+d)^{-p} h_n: n \in 
\mathbb{Z}^d_+\}$ is an orthonormal basis for $\Sc_p(\R^d)$, where $|n| = 
n_1+\cdots+n_d$. We also have (see
\cite[Appendix A.5, equation (A.26)]{MR562914})
\[\partial_i h_n =
\sqrt{\frac{n_i}{2}}h_{n-e_i}-\sqrt{\frac{n_i+1}{2}}h_{n+e_i}\]
and
\[x_i h_n(x) = \sqrt{\frac{n_i+1}{2}}h_{n+e_i}(x) 
+\sqrt{\frac{n_i}{2}}h_{n-e_i}(x),\]
with the convention that for a multi-index $n = (n_1,\cdots,n_d)$, if $n_i
< 0$ for some $i$, then $h_n \equiv 0$. Above recurrence implies that
$\partial_i,M_i:\Sc_{p}(\R^d)\to\Sc_{p-\frac{1}{2}}(\R^d)$ are bounded linear 
operators, 
for any $p \in 
\R$.\\
The operators $A,L$ are given as follows:
for $\phi \in 
\Sc(\R^d)$ and $x \in \R^d$,
\begin{equation}
\begin{cases}
A\phi := (A_1\phi,\cdots,A_r\phi),\\
A_i\phi(x) := \sum_{k=1}^d \sigma_{ki}(x)\partial_k\phi(x),\\
L\phi(x) := \frac{1}{2}\sum_{i,j=1}^d 
(\sigma\sigma^t)_{ij}(x)\partial^2_{ij}\phi(x)+\sum_{i=1}^d
b_i(x)\partial_i\phi(x),
\end{cases}
\end{equation}
where $\sigma^t$ denotes the transpose of $\sigma$. For $\psi \in 
\Sc'(\R^d)$ define the adjoint 
operators $A^{\ast}, L^\ast$ as follows.
\begin{equation}
\begin{cases}
A^{\ast}\psi := (A_1^{\ast}\psi,\cdots,A_r^{\ast}\psi),\\
A_i^{\ast}\psi := -\sum_{k=1}^d \partial_k\left(\sigma_{ki}\psi\right),\\
L^{\ast}\psi := 
\frac{1}{2}\sum_{i,j=1}^d \partial^2_{ij}\left((\sigma\sigma^t)_{ij}
\psi\right)-\sum_{i=1}^d
\partial_i\left(b_i\psi\right).
\end{cases}
\end{equation}
We now look at $A^{\ast},L^{\ast}$ as operators on $\Sc_p(\R^d)$.
\begin{thm} \label{AL-star-bnds}
Fix $p \in \R$. There exist constants $C_1=C_1(p),C_2=C_2(p) > 0$ such that
\[\|A^{\ast}_i\theta\|_{p-\frac{1}{2}} \leq 
C_1\|\theta\|_p,\,\|L^{\ast}\theta\|_{p-1} \leq 
C_2\|\theta\|_p,\, \forall \theta \in \Sc_{p}(\R^d).\]
Furthermore, we have the Monotonicity inequality for $(A^{\ast},L^{\ast})$, 
i.e. there exists a constant 
$C_p>0$ such that
\[2\inpr[p]{\theta}{L^{\ast}\theta}+\|A^{\ast}\theta\|_{HS(p)}^2 \leq 
C_p\|\theta\|_p^2,\, 
\forall \theta \in \Sc_{p+1}(\R^d),\]
where $\|A^{\ast}\theta\|_{HS(p)}^2:=\sum_{i=1}^d \|A^{\ast}_i\theta\|_{p}^2$.
\end{thm}
\begin{proof}
For any $q \in \R$, $\partial_i, M_i: \Sc_{q}(\R^d)\to 
\Sc_{q-\frac{1}{2}}(\R^d)$ 
are bounded linear 
operators. Using the 
definitions of $A^{\ast}$ and $L^{\ast}$, 
estimates on the norms 
follows.\\
Proof of the Monotonicity inequality for $(A^{\ast},L^{\ast})$ follows from 
\cite[Theorem 4.2 and Theorem 4.6]{dohs}.
\end{proof}

For $p > \frac{d}{4}$, let $R=R(p)$ be as in Proposition \ref{dirac-limit}. 
Then
\begin{equation}\label{Yt-sq-intg}
\Exp \, \|Y_t(\psi)\|_{-p}^2 \leq R^2 \left(\int_{\R^d} 
|\psi(x)| \,dx\right)^2< \infty.
\end{equation}

Using the norm estimates on $A^\ast_i, i=1,\cdots,d$ in the previous theorem 
and equation \eqref{Yt-sq-intg} the next result can be proved.
\begin{propn}\label{AY.B-is-mtgle}
Let $p > \frac{d}{4}$. Then $\{\int_0^t A^{\ast}(Y_s(\psi))\,.dB_s\}$ is an 
$(\F_t)$ adapted $\Sc_{-p-\frac{1}{2}}(\R^d)$ valued continuous martingale.
\end{propn}

Next result is analogous to \cite[Theorem 
3.3]{MR2373102}, except the uniqueness part.
\begin{thm}\label{SDE-Y}
Let $p > \frac{d}{4}$ and $\psi \in \Lone(\R^d)$. Then the 
$\Sc_{-p}(\R^d)$ valued 
continuous adapted process $\{Y_t(\psi)\}$ satisfies the 
following equation in $\Sc_{-p-1}(\R^d)$, a.s.
\begin{equation}\label{Astar-Lstar-eqn}
Y_t(\psi) = \psi + \int_0^t A^{\ast}(Y_s(\psi))\,.dB_s + 
\int_0^t L^{\ast}(Y_s(\psi))\,ds, \, \forall t \geq 0.
\end{equation}
This solution is also unique.
\end{thm}
\begin{proof}
By It\={o}'s formula for any $\phi \in \Sc(\R^d)$, and any $x \in \R^d$ we have
\begin{equation}\label{Ito-formula-implication}
(X_t(\phi))(x)= \phi(x)+\int_0^t (X_s(A\phi))(x).\,dB_s+\int_0^t 
(X_s(L\phi))(x)\,ds.
\end{equation}
Since $L\phi \in \Sc(\R^d)$, $\{x\mapsto 
(X_t(L\phi))(x)\}$ is an $\Sc(\R^d)$ valued process. Using differentiation 
under the sign of integration we can establish the existence of all derivatives 
of $x \mapsto \int_0^t 
(X_s(L\phi))(x)\,ds$ for any $t \geq 0$. Given non-negative 
integers $N, \alpha_1,\cdots, \alpha_d$, the terms 
\[\sup_{x\in\R^d}(1+|x|^2)^N|\partial_1^{\alpha_1}\cdots\partial_d^{\alpha_d}
(X_s(L\phi))(x)|, s \in [0,t]\]
can be dominated uniformly in $s$ and hence $\{x\mapsto\int_0^t 
(X_s(L\phi))(x)\,ds\}$ is an $\Sc(\R^d)$ valued process. From 
\eqref{Ito-formula-implication}, we conclude that $\{x\mapsto\int_0^t 
(X_s(A\phi))(x).\,dB_s\}$ is an $\Sc(\R^d)$ valued 
process, since the other terms are so. Then for $\phi \in \Sc(\R^d)$, using 
\eqref{YasX*} we can show a.s. $t \geq 0$
\[\inpr{Y_t(\psi)}{\phi} =\inpr{\psi+\int_0^t 
A^{\ast}Y_s(\psi).\,dB_s+\int_0^t 
L^{\ast}Y_s(\psi)\,ds}{\phi}.\]
The argument for above equality is similar to \cite[Theorem 
3.3]{MR2373102}. This proves $Y_t(\psi)$ solves 
\eqref{Astar-Lstar-eqn} in $\Sc_{-p-1}(\R^d)$.\\
To prove the uniqueness, let $\{Y^1_t\}, \{Y^2_t\}$ be two 
solutions of \eqref{Astar-Lstar-eqn} and let 
$Z_t:=Y^1_t-Y^2_t, t \geq 0$. Then using It\={o} formula for 
$\|\cdot\|^2_{-p-1}$ we obtain 
a.s. for all $t \geq 0$
\[\|Z_t\|^2_{-p-1} = \int_0^t 
\left[2\inpr[-p-1]{Z_s}{L^{\ast}Z_s}+\sum_{i=1}^d\|A^{\ast}_iZ_s\|_{-p-1}
^2\right] ds + M_t,\]
where $\{M_t\}$ is some continuous local martingale with $M_0=0$. Now using the 
Monotonicity inequality (Theorem \ref{AL-star-bnds}) and the Gronwall's 
inequality we can show $ Z_t = 0, t \geq 0$ a.s., 
which shows a.s. $Y^1_t = Y^2_t, t \geq 0$. The argument is similar to 
\cite[Lemma 3.6]{MR3063763}.
\end{proof}

We now prove the existence and uniqueness of 
solution to \eqref{Lstar-ivp} with initial condition $\psi \in \Lone(\R^d)$. 
This result is analogous to \cite[Theorems 4.3, 4.4]{MR2373102}.
\begin{thm}\label{ivp-soln}
Let $p > \frac{d}{4}$ and $\psi \in \Lone(\R^d)$. Then $\psi(t):=\Exp\, 
Y_t(\psi)$ solves the 
initial value problem 
\eqref{Lstar-ivp}, i.e.
\[\Exp\, Y_t(\psi) = \psi + \int_0^t 
L^{\ast}\,\Exp \,Y_s(\psi)\,ds\]
holds in $\Sc_{-p-1}(\R^d)$. Furthermore this is the unique solution.
\end{thm}
\begin{proof}
The arguments are similar to \cite[Theorems 4.3, 4.4]{MR2373102}. Taking term 
by 
term expectation on 
both sides of \eqref{Astar-Lstar-eqn} and using Proposition 
\ref{AY.B-is-mtgle}, we can 
show $\psi(t):=\Exp\, 
Y_t(\psi)$ solves \eqref{Lstar-ivp} in $\Sc_{-p-1}(\R^d)$. Here $L^{\ast}\Exp 
\,Y_s(\psi) = \Exp \,L^{\ast}
\,Y_s(\psi)$, since $L^\ast$ is a bounded linear operator from 
$\Sc_{-p}(\R^d)$ to $\Sc_{-p-1}(\R^d)$.\\
The proof of uniqueness of the solution is same as in \cite[Theorem 
4.4]{MR2373102}. We use the Monotonicity inequality in 
Theorem \ref{AL-star-bnds} and 
the Gronwall's inequality. \end{proof}

\begin{rem}
In \cite[Theorem 
4.4]{MR2373102} the Monotonicity inequality was in the hypothesis, whereas in 
our case it has been proved.
\end{rem}

The process $\{Y_t(\psi)\}$ can also be described in terms of $\{X(t,0)\}$ 
without using the integral representation \eqref{defn-Y}. We show that the 
tempered 
distribution 
$Y_t(\psi)(\omega)$ is given by an 
integrable function. This representation 
of $\{Y_t(\psi)\}$ is 
similar to the representation obtained in \cite[Lemma 3.6]{MR3063763}, where 
the 
author looked at the solution of stochastic partial differential equations 
governed by certain non-linear operators. \begin{propn}\label{dist-Y}
Let $\psi \in \Lone(\R^d)$ and $\omega \in \mathcal{N}$. Then
\[Y_t(\psi) = e^{-t\,tr(C)}\,\tau_{Z_t}\psi_t(\cdot),\]
where $Z_t := X(t,0), \psi_t(x):=\psi(e^{-tC}x)$ for $t \geq 0, x 
\in 
\R^d$ and $tr(C)$ is the trace of the matrix $C$. 
\end{propn}
\begin{proof}
For any $\phi \in \Sc(\R^d)$, we have
\begin{align*}
\inpr{Y_t(\psi)}{\phi} &= \int_{\R^d} \psi(x) \phi(X(t,x)) \, dx = 
\int_{\R^d} \psi(x) \phi(e^{tC}x + Z_t) \, dx\\
&= |det(e^{-tC})|\int_{\R^d} \psi(e^{-tC}(z-Z_t)) \phi(z) \, dz,\, 
(\text{putting } z = 
e^{tC}x+Z_t)\\
&= e^{-t\,tr(C)} \int_{\R^d} (\tau_{Z_t}\psi_t)(z) \phi(z) \, dz.
\end{align*}
The equality $|det(e^{-tC})| = e^{-t\,tr(C)}$ follows from 
\cite[Problem 5.6.P43]{MR2978290}. Hence $Y_t(\psi) = 
e^{-t\,tr(C)}\,\tau_{Z_t}\psi_t(\cdot)$.
\end{proof}
\begin{rem}
We obtained probabilistic representation of solutions of 
\eqref{Lstar-ivp}, when the initial condition $\psi \in \Lone(\R^d)$ and the 
coefficients $\sigma, b$ are in a specific form. Possible extensions of these 
results 
to the case - when the initial 
condition $\psi \in \mathcal{L}^q(\R^d)$ for some $q > 1$ or more 
generally a finite linear combination of the distributional derivatives of 
$\mathcal{L}^q(\R^d)$ functions ($q \geq 1$) - will be taken up as future work. 
\end{rem}

\noindent\textbf{Acknowledgement:} The author would like to thank Professor B.
Rajeev, Indian Statistical Institute, Bangalore for valuable suggestions during
the work.

%\bibliographystyle{amsplain}
%\bibliography{inv-mre-ref.bib}

\end{document}